\documentclass[a4paper,12pt]{article}
\usepackage[utf8]{inputenc}
\usepackage[english]{babel}
\usepackage{latexsym,amssymb,amsmath,amsthm,amsfonts,euscript}

\newtheorem{theorem}{Theorem}
\newtheorem{prop}[theorem]{Proposition}

\title{Tolstov's Theorem in the Commutative Banach Algebra $\mathbb{A}_3$}

\author{M.\,V.~Tkachuk\\
Institute of Mathematics \\ of the National Academy of Sciences of Ukraine\\
maxim.v.tkachuk@gmail.com}

\begin{document}

\maketitle

\vskip 5mm

\rightline{UDC 517.54, MSC 30G35}

\begin{abstract}
In this paper, the conditions of monogenicity are weakened for functions
with values in a three-dimensional commutative algebra over the field of
complex numbers. By monogenicity we mean continuity and the existence of
the G\^ateaux derivative.
\end{abstract}

\textbf{1. Introduction.}
In the algebra of complex numbers $\mathbb{C}$, a function
$F \colon \mathbb{C} \to \mathbb{C}$ is called monogenic at a point
$\xi_0 \in \mathbb{C}$ if the finite limit
\begin{equation}\label{compl-der}
\lim_{\xi\to\xi_0}
\frac{F(\xi)-F(\xi_0)}{\xi-\xi_0}
\end{equation}
exists. This limit is called the derivative of $F$ at the point $\xi_0$.
A function that is monogenic at every point of a domain
$D \subset \mathbb{C}$ is called holomorphic in this domain
(see \cite{Goursa}).

The problem of weakening the conditions of monogenicity for functions of
a complex variable has been studied by
H.~Bohr \cite{Bohr},
H.~Rademacher \cite{Rademacher},
D.\,E.~Menshov \cite{menshov-1},
V.\,S.~Fedorov \cite{fedorov},
G.\,P.~Tolstov \cite{Tolstov},
Yu.\,Yu.~Trokhimchuk \cite{Trokhimchuk,zb_trokhinchuk},
G.\,Kh.~Sindalovskii \cite{Sindalovski},
D.\,S.~Telyakovskii \cite{Teliakovski},
E.\,P.~Dolzhenko \cite{Dolgenko},
and M.\,T.~Brodovich \cite{Brodovich}.

The purpose of this paper is to weaken the conditions of monogenicity
for functions taking values in commutative banach algebra $\mathbb{A}_3$ over the field of complex numbers by generalizing Tolstov's theorem.

\begin{theorem}[G.~Tolstov \cite{Tolstov}]
Let
\[
f(z)=u(x,y)+iv(x,y)
\]
be a locally bounded function of the complex variable $z=x+iy$ in a
domain $D \subset \mathbb{C}$. Suppose that the partial derivatives
\[
\frac{\partial u}{\partial x},\quad
\frac{\partial u}{\partial y},\quad
\frac{\partial v}{\partial x},\quad
\frac{\partial v}{\partial y}
\]
exist everywhere in $D$ and satisfy the Cauchy--Riemann equations almost
everywhere:
\[
\frac{\partial u}{\partial x}
=
\frac{\partial v}{\partial y},
\qquad
\frac{\partial u}{\partial y}
=
-\frac{\partial v}{\partial x}.
\]
Then the function $f(z)$ is analytic in $D$.
\end{theorem}

Note that the Cauchy--Riemann equations can be written in the form
\[
\frac{\partial f}{\partial y}
=
i\frac{\partial f}{\partial x}.
\]

\textbf{2. Structure of the three-dimensional commutative Banach algebra
$\mathbb{A}_3$.}

Consider the three-dimensional commutative associative Banach algebra
$\mathbb{A}_3$ with identity $1$ over the field $\mathbb{C}$, whose basis
is the triple $\{1,\rho,\rho^2\}$, where
\[
\rho^3 = 0.
\]
Note that this basis is a Cartan basis \cite{cartan}.

Define the Euclidean norm of an element of the algebra by
\[
\|a+b\rho+c\rho^2\|
:=
\sqrt{|a|^2+|b|^2+|c|^2},
\qquad
a,b,c\in\mathbb{C}.
\]

The algebra $\mathbb{A}_3$ has a unique maximal ideal
\[
\mathcal{I}
:=
\{\lambda_1\rho+\lambda_2\rho^2
:
\lambda_1,\lambda_2\in\mathbb{C}\},
\]
which is also its radical.

Consider the linear mapping
$f:\mathbb{A}_3\to\mathbb{C}$ defined by
\begin{equation}\label{mult-func}
f(a+b\rho+c\rho^2)=a.
\end{equation}
Since the kernel of $f$ is the maximal ideal $\mathcal{I}$, the mapping
$f$ is a continuous multiplicative functional
(see \cite[p.~135]{Hil_Filips}).

Fix a real three-dimensional subspace
\[
E_3
=
\{\zeta=xe_1+ye_2+ze_3 : x,y,z\in\mathbb{R}\}
\subset \mathbb{A}_3,
\]
where the vectors $e_1,e_2,e_3$ are linearly independent over the field
of real numbers $\mathbb{R}$, but do not necessarily form a basis of
the algebra $\mathbb{A}_3$. We impose only one condition on the choice
of the subspace $E_3$:
\[
f(E_3)=\mathbb{C},
\]
that is, the image of $E_3$ under $f$ is the entire complex plane
(see \cite{Pukh-5,Sh-co}).

\textbf{3. Monogenic functions in the algebra $\mathbb{A}_3$.}

It is well known that there are different notions of differentiability
for mappings in normed linear spaces. In particular, one uses strong
Fr\'echet differentiability and weak G\^ateaux differentiability
(see, for example, \cite{Hil_Filips}), where the corresponding
Fr\'echet and G\^ateaux derivatives are defined as linear operators.

Using the G\^ateaux differential,
I.\,P.~Mel'nichenko \cite{Mel'nichenko75} proposed to consider the
G\^ateaux derivative also as a function defined on the same domain
as the original function.

Let $\Phi:\Omega\to\mathbb{A}_3$ be a function defined in a domain
$\Omega\subset E_3$. If for every point $\zeta\in\Omega$ there exists
an element $\Phi_G'(\zeta)\in\mathbb{A}_3$ such that
\begin{equation}\label{eq:deriv}
\lim_{\delta\to 0+0}
\bigl(
\Phi(\zeta+\delta h)-\Phi(\zeta)
\bigr)\delta^{-1}
=
h\Phi_G'(\zeta)
\qquad
\forall h\in E_3,
\end{equation}
then the function
$\Phi_G':\Omega\to\mathbb{A}_3$
is called the \emph{G\^ateaux derivative} of $\Phi$.

A function $\Phi:\Omega\to\mathbb{A}_3$ is called
\emph{monogenic} in the domain $\Omega\subset E_3$
if $\Phi$ is continuous and has a G\^ateaux derivative at every point
of $\Omega$
(see
\cite{shpakivskyi_plaksa_umzh,overview_plaksa,Plaksa_UMB}).

In \cite{Pl-zb17}, one of the conditions of monogenicity was weakened.
Namely, it was shown that if the G\^ateaux derivative of a function
$\Phi:\Omega\to\mathbb{A}_3$ exists at all points of a domain
$\Omega\subset E_3$, then continuity of $\Phi$ may be replaced by local
boundedness in $\Omega$.

\begin{prop}\label{prop:last_component}
Suppose that all intersections of the domain
$\Omega \subset \mathbb{A}_3$ with hyperplanes parallel to the radical
$\mathcal{I}$ of the algebra $\mathbb{A}_3$ are connected. Let
\[
\Phi(\zeta)=\rho^2\Phi_2(\zeta)
\]
be a locally bounded function satisfying \eqref{eq:deriv} for the vectors
\[
\{\pm 1,\ \pm i,\ \pm \rho,\ \pm i\rho,\ \pm \rho^2,\ \pm i\rho^2\}.
\]
Then
\[
\Phi_2(\zeta)=F_2(f(\zeta))=F_2(z),
\]
where $F_2:D\to\mathbb{C}$ is an analytic function in the domain
$D=f(\Omega)$.
\end{prop}

\begin{proof}
The left-hand side of \eqref{eq:deriv} has the form $a\rho^2$.
Therefore, $h\Phi'(\zeta)$ and $\Phi'(\zeta)$ also have this form
(the latter follows by taking $h=1$). Let
\[
\Phi'(\zeta)=\Phi_2'(\zeta)\rho^2.
\]
Substituting this into \eqref{eq:deriv}, we obtain
\begin{equation}
\lim_{\varepsilon\to 0+0}
\bigl(
\Phi_2(\zeta+\varepsilon h)-\Phi_2(\zeta)
\bigr)
\varepsilon^{-1}\rho^2
=
h\,\Phi_2'(\zeta)\rho^2.
\label{eq:deriv_k}
\end{equation}

If
\[
h\in
\{\pm \rho,\ \pm i\rho,\ \pm \rho^2,\ \pm i\rho^2\}
\subset \mathcal{I},
\]
then $h\rho^2=0$. Hence,
\[
\lim_{\varepsilon\to 0+0}
\bigl(
\Phi_2(\zeta+\varepsilon h)-\Phi_2(\zeta)
\bigr)
\varepsilon^{-1}
=
0.
\]
Thus, the directional derivative in the direction $h$ vanishes:
\[
\frac{d\Phi_2}{dh}=0.
\]

If all intersections of $\Omega$ with hyperplanes parallel to the
radical are connected, then the function $\Phi_2(\zeta)$ is constant
on each such intersection. Therefore, the function
\[
F_2(z)=F_2(f(\zeta))=\Phi_2(\zeta)
\]
is well defined.

Substituting $h=\pm 1$ and $h=\pm i$ into \eqref{eq:deriv_k}, we obtain
\[
\lim_{\varepsilon\to 0+0}
\bigl(
\Phi_2(\zeta\pm\varepsilon)-\Phi_2(\zeta)
\bigr)
\varepsilon^{-1}
=
\pm \Phi_2'(\zeta),
\]
\[
\lim_{\varepsilon\to 0+0}
\bigl(
\Phi_2(\zeta\pm i\varepsilon)-\Phi_2(\zeta)
\bigr)
\varepsilon^{-1}
=
\pm i\Phi_2'(\zeta).
\]
Hence,
\[
\lim_{\varepsilon\to 0}
\bigl(
\Phi_2(\zeta+i\varepsilon)-\Phi_2(\zeta)
\bigr)
\varepsilon^{-1}
=
i
\lim_{\varepsilon\to 0}
\bigl(
\Phi_2(\zeta+\varepsilon)-\Phi_2(\zeta)
\bigr)
\varepsilon^{-1}.
\]
Equivalently,
\[
\lim_{\varepsilon\to 0}
\bigl(
F_2(z+i\varepsilon)-F_2(z)
\bigr)
\varepsilon^{-1}
=
i
\lim_{\varepsilon\to 0}
\bigl(
F_2(z+\varepsilon)-F_2(z)
\bigr)
\varepsilon^{-1},
\]
that is,
\[
\frac{\partial F_2}{\partial y}
=
i\frac{\partial F_2}{\partial x}.
\]

Therefore, by Tolstov's theorem \cite{Tolstov}, the function $F_2(z)$
is analytic.
\end{proof}

Note that if the function
\[
\Phi(\zeta)
=
\rho\Phi_1(\zeta)+\rho^2\Phi_2(\zeta)
\]
or
\[
\Phi(\zeta)
=
\Phi_0(\zeta)+\rho\Phi_1(\zeta)+\rho^2\Phi_2(\zeta),
\]
then the function $\rho\Phi(\zeta)$ (or, respectively,
$\rho^2\Phi(\zeta)$) satisfies the previous proposition.
This implies the following statement.

\begin{prop}\label{prop:first_component}
Suppose that all intersections of the domain
$\Omega \subset \mathbb{A}_3$ with hyperplanes parallel to the radical
$\mathcal{I}$ of the algebra $\mathbb{A}_3$ are connected, and let
$\Phi:\Omega\to\mathbb{A}_3$ be a locally bounded function satisfying
\eqref{eq:deriv} for the vectors
\[
\{\pm 1,\ \pm i,\ \pm \rho,\ \pm i\rho,\ \pm \rho^2,\ \pm i\rho^2\}.
\]
Then the first nonzero component $\Phi_*$ of the function $\Phi$ can be
represented in the form
\[
\Phi_*(\zeta)=F(f(\zeta))=F(z),
\]
where $F:D\to\mathbb{C}$ is an analytic function in the domain
$D=f(\Omega)$.
\end{prop}

In \cite{plaksa_shpakivskii_finite_algebra} it was shown that the
principal extension of an analytic function $F(z)$ to a Banach algebra
can be written in the form
\[
\frac{1}{2\pi i}
\int_{\gamma}
F(t)(t-\zeta)^{-1}\,dt,
\]
where $\gamma$ is a closed rectifiable Jordan curve in
$D=f(\Omega)\subset\mathbb{C}$ enclosing the point $f(\zeta)$.
The principal extension is G\^ateaux differentiable, and each component
of its expansion in the Cartan basis is an analytic function of the
variable $f(\zeta)$. Moreover, the first component of this expansion is
equal to $F(f(\zeta))$.

\begin{theorem}
Suppose that all intersections of the domain
$\Omega \subset \mathbb{A}_3$ with hyperplanes parallel to the radical
of the algebra are connected, and let
$\Phi:\Omega\to\mathbb{A}_3$ be a locally bounded function satisfying
\eqref{eq:deriv} for the vectors
\[
\{\pm 1,\ \pm i,\ \pm \rho,\ \pm i\rho,\ \pm \rho^2,\ \pm i\rho^2\}.
\]
Then the function $\Phi$ admits the representation
\[
\Phi(\zeta)
=
\frac{1}{2\pi i}
\int_{\gamma}
\bigl(
F_0(t)+F_1(t)\rho+F_2(t)\rho^2
\bigr)
(t-\zeta)^{-1}\,dt,
\]
where $F_0$, $F_1$, and $F_2$ are analytic functions in the domain
$D=f(\Omega)$, and $\gamma$ is a closed rectifiable Jordan curve in
$D=f(\Omega)\subset\mathbb{C}$ enclosing the point $f(\zeta)$.
\end{theorem}

\begin{proof}
The proof is based on Proposition~\ref{prop:first_component}
and on the fact that the first component of the principal extension
of an analytic function coincides with the function itself.

Let
\[
\Phi(\zeta)
=
\Phi_0(\zeta)+\Phi_1(\zeta)\rho+\Phi_2(\zeta)\rho^2.
\]
Then, by Proposition~\ref{prop:first_component},
\[
\Phi_0(\zeta)=F_0(z),
\]
where $F_0$ is an analytic function of the complex variable $z$.

The principal extension of $F_0(z)$ to the algebra has first component
equal to $F_0(z)$. Therefore, the difference
\[
\Phi(\zeta)
-
\frac{1}{2\pi i}
\int_{\gamma}
F_0(t)(t-\zeta)^{-1}\,dt
=
\Phi_{11}(\zeta)\rho+\Phi_{12}(\zeta)\rho^2
\]
is also locally bounded and satisfies \eqref{eq:deriv} for the vectors
\[
\{\pm 1,\ \pm i,\ \pm \rho,\ \pm i\rho,\ \pm \rho^2,\ \pm i\rho^2\}.
\]

Applying this argument iteratively, we obtain the desired representation.
\end{proof}

Consider now the real three-dimensional linear subspace
$E_3 \subset \mathbb{A}_3$ defined above, for which
\[
f(E_3)=\mathbb{C}.
\]
The preimage of each point $z \in \mathbb{C}$ is a straight line
\[
f^{-1}(z)\subset E_3.
\]

A basis of the space $E_3$ can be obtained by choosing nonzero vectors
\[
a \in f^{-1}(1), \qquad
b \in f^{-1}(i), \qquad
c \in f^{-1}(0)
\]
in the corresponding preimages.

The proofs of the following theorems are completely analogous to the
proofs of the previous statements.

\begin{theorem}
Let $\Omega \subset E_3$ be a domain whose intersections with lines
parallel to the line $f^{-1}(0)$ are connected. Let
$\Phi:\Omega\to\mathbb{A}_3$ be a locally bounded function satisfying
\eqref{eq:deriv} for the vectors
\[
\{\pm a,\ \pm b,\ \pm c\}.
\]
Then the first nonzero component $\Phi_*$ of the function $\Phi$
can be represented in the form
\[
\Phi_*(\zeta)=F(f(\zeta))=F(z),
\]
where $F:D\to\mathbb{C}$ is an analytic function in the domain
\[
D=f(\Omega).
\]
\end{theorem}

\begin{theorem}
Let $\Omega \subset E_3 \subset \mathbb{A}_3$ be a domain whose
intersections with lines parallel to the line $f^{-1}(0)$ are
connected. Let $\Phi:\Omega\to\mathbb{A}_3$ be a locally bounded
function satisfying \eqref{eq:deriv} for the vectors
\[
\{\pm a,\ \pm b,\ \pm c\}.
\]
Then the function $\Phi$ admits the representation
\[
\Phi(\zeta)
=
\frac{1}{2\pi i}
\int_{\gamma}
\bigl(
F_0(t)+F_1(t)\rho+F_2(t)\rho^2
\bigr)
(t-\zeta)^{-1}\,dt,
\]
where $F_0$, $F_1$, and $F_2$ are analytic functions in the domain
\[
D=f(\Omega),
\]
and $\gamma$ is a closed rectifiable Jordan curve in
$D=f(\Omega)\subset\mathbb{C}$ enclosing the point $f(\zeta)$.
\end{theorem}

\textbf{Funding Acknowledgments.}
This work was supported by a grant from the
Simons Foundation (00014586, M.\,V.~T.).

\end{document}